\documentclass{amsart}
\usepackage{amsmath,amsthm}
\usepackage{geometry}
\usepackage{graphicx}
\usepackage{subfigure}
\usepackage{amssymb,mathabx}
\usepackage[usenames,dvipsnames]{color}
\usepackage[numbers]{natbib}
\let\cite=\citet

\newtheorem{theorem}{Theorem}[section]
\newtheorem{lemma}[theorem]{Lemma}

\usepackage[notcite,notref]{showkeys} 

\usepackage{hyperref}

\usepackage{MACROS/remarks,MACROS/mydef,MACROS/boldfonts,xspace}

\begin{document}

\title[Artificial compressibility revisited]{High-order time stepping for
the Navier-Stokes equations with minimal computational complexity}

\author[J.-L.~Guermond]{Jean-Luc Guermond$^1$} \address{$^1$Department
  of Mathematics, Texas A\&M University 3368 TAMU, College Station, TX
  77843-3368, USA.  On leave from CNRS, France.}
\email{guermond@math.tamu.edu}

\author[P.D.~Minev ]{Peter D. Minev$^2$} \address{$^2$Department of
  Mathematical and Statistical Sciences, University of Alberta,
  Edmonton, Alberta Canada T6G 2G1.}  \email{minev@ualberta.ca}

\thanks{This material is based upon work supported by the National
  Science Foundation grants DMS-0713829, by the Air Force Office of
  Scientific Research, USAF, under grant/contract number
  FA9550-09-1-0424, and a Discovery grant of the
  National Science and Engineering Research Council of Canada. This
  publication is also partially based on work supported by Award
  No. KUS-C1-016-04, made by King Abdullah University of Science and
  Technology (KAUST).}




\begin{abstract}
  In this paper we present extensions of the schemes proposed in
  \cite{GM14} that lead to a decoupling of the velocity components in
  the momentum equation. The new schemes reduce the solution of the
  incompressible Navier-Stokes equations to a set of classical
  uncoupled parabolic problems for each Cartesian component of the velocity.
  The pressure is explicitly recovered after the velocity is computed.
\end{abstract}

\maketitle

\section{Introduction}
In \cite{GM14}, we considered the possibility to construct high order
artificial compressibility schemes for incompressible flow.  The
resulting schemes require the solution of problems of the type
$\bu - \nu \dt \LAP \bu - \dt \GRAD \DIV \bu = \dt \bef$, $\dt$ being
the time step.  The corresponding discrete problem clearly has a
condition number of the order of $\dt h^{-2}$, $h$ being the spatial
step.  In this paper we consider some possibilities to improve this
algorithm by discretizing the $\GRAD \DIV$ operator in an
implicit-explicit fashion in order to decouple the Cartesian
components of the velocity and thereby reducing the problem to a
series of scalar-valued parabolic problems. In fact, such
strategies based on the direction splitting approach, which was popular at
that time, have been proposed in the literature in the 1960s and
70s.   For instance, a direction splitting scheme
that includes the splitting of the $\GRAD \DIV$ operator has been
proposed in the Russian literature by the groups of Yanenko (see \cite{Yanenko_1966},
\cite{MR46:6613}, section 8.3) and Ladizhenskaya ( see
\cite{Ladhyz70b}, chapter VI, section 9.2, and the references
therein).  In the Western literature, such schemes have been proposed and
analyzed by \cite{Te77}, chapter III, section 8.3.  In the present paper we
generalize the approach to make it applicable to non-Cartesian
grids without splitting, and we combine it with the defect correction
approach discussed in \cite{GM14} to increase the order.
Furthermore, we propose new direction splitting schemes that allow the
use of direct methods for three dimensional problems.

\section{Preliminaries}\label{sec:prelim}
\subsection{Formulation of the problem}
We consider the time-dependent Navier-Stokes equations on a finite
time interval $[0,T]$ and in a domain $\Omega$ in $\Real^d$ with a Lipschitz
boundary. Since the nonlinear term in the Navier-Stokes equations has
no significant influence on the pressure-velocity coupling and since
this term is usually made explicit, we henceforth mostly consider the
time-dependent Stokes equations written in terms of velocity $\ue$ and
pressure $\pe$:
\begin{equation}
\left\{
\begin{aligned}
&\partial_t \ue + \polA \ue+ \GRAD \pe = \bef
                  \quad \text{in $\Omega\times[0,T]$}, \\
&\DIV \ue  =  0    \quad \text{in $\Omega\times[0,T]$},\\
&\ue|_{\partial \Omega}  =  0 \quad \text{in [0,T]}, \quad \text{ and }
\ue|_{t=0} =  \ue_0 \quad \text{in $\Omega$},
\end{aligned}\right. \label{NS}
\end{equation}
where $\bef$ is a smooth source term and $\ue_0$ is a solenoidal
initial velocity field with zero normal trace at the boundary of
$\Omega$.  The operator $\polA$ is assumed to be linear, $\bH^1$-coercive
and bounded, \ie there are two constants $\nu>0$ and $M <\infty$ such
that $\int_\Omega \polA\bu\SCAL\bu \diff\bx \ge \nu \|\bu\|_{\bH^1(\Omega)}^2$ and
$|\int_\Omega \polA\bu\SCAL\bv \diff\bx |\le M
\|\bu\|_{\bH^1(\Omega)}\|\bv\|_{\bH^1(\Omega)}$, for all $\bu,\bv\in
\bH^1_0(\Omega)$.  For the sake of simplicity, we consider homogeneous
Dirichlet boundary conditions on the velocity.

We are going to be mainly concerned with time discretizations of the
above problem. Let $\dt>0$ be a time step and set $t^n=n\dt$ for
$0\leq n\leq N=\floor{T/\dt}$, where $\floor{\cdot}$ is the floor
function.  Let $\phi^0,\phi^1,\ldots \phi^N$ be some sequence of
functions in a Hilbert space $E$.  We denote by $\phi_{\dt}$ this
sequence, and we define the following discrete norms:
$\|\phi_{\dt}\|_{\ell^2(E)} := \big(\dt \sum_{n=0}^N
\|\phi^n\|_{E}^2\big)^{\frac12}$, $\|\phi_{\dt}\|_{\ell^\infty(E)} :=
\max_{0\leq n \leq N} \left(\|\phi^n\|_{E}\right)$.  In addition, we
denote the first differences of the elements of the sequence by
$\displaystyle \delta_t \phi^n= (\phi^n-\phi^{n-1}), n=1,\dots,N$, and their average 
$\displaystyle \bar{\phi}^n= (\phi^n+\phi^{n-1})/2, n=1,\dots,N$.  
The sequences $\delta_t  \phi^1,\ldots \delta_t  \phi^N$ and $\bar{\phi}^1,\ldots \bar{\phi}^N$
are denoted by $\delta_t \phi_{\dt}$ and $\bar{\phi}_{\dt}$ correspondingly. We also
denote by $c$ a generic constant that is independent of $\dt$ and $\epsilon$ but
possibly depends on the data, the domain, and the solution.

\subsection{High-order artificial compressibility}
In \cite{GM14} we introduced a series of second and third-order
schemes based on the following elementary first-order artificial
compressibility algorithm:
\begin{equation}
\frac{\bu^{n+1}-\bu^n}{\dt} + \polA\bu^{n+1} + \GRAD p^{n+1} = \bef^{n+1},
\quad \frac{\epsilon}{\dt}(p^{n+1}-p^n) +\DIV\bu^{n+1} = 0,
\label{first_order_art_comp}
\end{equation}
where $\epsilon>0$ is a user-dependent parameter that is usually
chosen to be proportional to $\dt$, \ie $\epsilon = \dt/\chi$ where $\chi$ is
of order one. One interesting property of this scheme is that it
decouples the velocity and the pressure; more precisely, the algorithm
can be recast as follows:
 \begin{equation}
\frac{\bu^{n+1}-\bu^n}{\dt} + \polA\bu^{n+1} -\chi \GRAD \DIV \bu^{n+1}
= \bef^{n+1} - \GRAD p^{n} ,
\quad p^{n+1} = -p^n  -\chi\DIV\bu^{n+1}. 
\label{first_order_art_comp_p_eliminated}
\end{equation}

The above algorithm has been extended to third-order accuracy in time
in \citep{GM14} by using a defect correction method. Denoting by
$B\bu$ the nonlinear term in the Navier-Stokes equations, the full
third order scheme is as follows:
\begin{equation}  n \ge 0, \qquad
\begin{cases}
  \textbf{nl}_0^{n+1}=B \bu_0^n, \\ \displaystyle
  \frac{\bu_0^{n+1}-\bu_0^n}{\dt} + \polA\bu_0^{n+1}  
-\chi\GRAD\DIV \bu_0^{n+1}+ \GRAD p_0^{n}  = \bef^{n+1} - \textbf{nl}_0^{n+1}\\
  p_0^{n+1} = p_0^{n}  - \chi \DIV\bu_0^{n+1},\\
  d\bu_0^{n+1} = (\bu_0^{n+1}-\bu_0^n)/\dt, \quad dp_0^{n+1} =
  (p_0^{n+1} -p_0^n)/\dt
\end{cases} \label{defect_ns_bootstrap0}
\end{equation}
\begin{equation}  n \ge 1, \qquad
\begin{cases}
d^2\bu_0^{n+1}= (d\bu_0^{n+1}-d\bu_0^{n})/\dt, \\
\textbf{nl}_1^n=B (\bu_0^{n} + \tau \bu_1^{n-1}), \\ \displaystyle
\frac{\bu_1^{n}-\bu_1^{n-1}}{\dt} + \polA\bu_1^{n}  -\chi\GRAD\DIV
\bu_1^n + \GRAD (p_1^{n-1} + dp_0^{n} )
\displaystyle
= -\frac12 d^2\bu_0^{n+1}
- \frac{\textbf{nl}_1^n - \textbf{nl}_0^n}{\dt},\\
p_1^{n} = p_1^{n-1}  + dp_0^{n}  - \chi  \DIV\bu_1^{n},\\
d\bu_1^{n} = (\bu_1^{n}-\bu_1^{n-1})/\dt, 
\quad dp_1^{n} = (p_1^{n}-p_1^{n-1})/\dt,
\end{cases}\label{defect_ns_bootstrap1}
\end{equation}
\begin{equation}  n \ge 2, \quad
\begin{cases}
d^2\bu_1^{n}=(d\bu_1^{n}-d\bu_1^{n-1})/\dt, \qquad
d^3\bu_0^{n+1}=(d^2\bu_0^{n+1}-d^2\bu_0^{n})/\dt, \\
\textbf{nl}_2^{n-1}=B (\bu_0^{n-1} + \tau \bu_1^{n-1} +\tau^2 \bu_2^{n-2})\\ \displaystyle
\frac{\bu_2^{n-1}-\bu_2^{n-2}}{\dt} + \polA\bu_2^{n-1} -\chi\GRAD \DIV
\bu_2^{n-1} + \GRAD (p_2^{n-2} + dp_1^{n-1}) \\ \displaystyle
\hspace{4cm}=  -\frac12 d^2\bu_1^{n}  + \frac16 d^3\bu_0^{n+1} 
 -  \frac{\textbf{nl}_2^{n-1} - \textbf{nl}_1^{n-1}}{\dt^2}\\
p_2^{n-1} = p_2^{n-2} + dp_1^{n-1} - \chi \DIV\bu_2^{n-1},\\
\bu^{n-1} =  \bu_0^{n-1} + \tau \bu_1^{n-1} +\tau^2 \bu_2^{n-1}, \quad
p^{n-1} =  p_0^{n-1} + \tau p_1^{n-1} +\tau^2 p_2^{n-1},
\end{cases}\label{defect_ns_bootstrap2}
\end{equation}
The stage \eqref{defect_ns_bootstrap0} yields a first order
approximation of the velocity and the pressure, the second stage
\eqref{defect_ns_bootstrap1} yields a second order approximation of
the velocity and the pressure, and the third stage
\eqref{defect_ns_bootstrap2} yields a third order approximation of the
velocity and the pressure.

One drawback of the above scheme is the presence of the $\GRAD \DIV$
operator since this operator couples all the Cartesian components of
the velocity and can lead to locking if not discretized properly.  In
the next section we introduce a first order artificial compressibility
scheme that decouples the different components of the velocity, \ie we
develop a decoupled version of the first stage
\eqref{defect_ns_bootstrap0}.  We will use this approach later in the
paper to modify the subsequent two stages and create a high-order time
stepping for the Navier-Stokes equations that requires only the
solution of a set of scalar-valued parabolic problems for each Cartesian
component of the velocity.  Since the proofs of stability of these
schemes in two and three dimensions differ somewhat, we will consider these
two cases separately.

\section{Splitting of the grad-div operator}\label{sec:grad_div}
\subsection{Splitting of $\polA$}
To be general we are going to assume that the operator $\polA$ admits
the following decomposition $\polA\bu = Au - \GRAD(\lambda \DIV \bu)$
where $\lambda$ is a smooth positive scalar field.  We assume also
that $A$ is block diagonal, $\bH^1$-coercive and bounded, \ie
$A\bu = (A_1 u_1,\ldots,A_d u_d)\tr$,
$\int_\Omega A\bu\SCAL\bu \diff\bx \ge \nu \|\bu\|_{\bH^1(\Omega)}^2$
and
$|\int_\Omega A\bu\SCAL\bv \diff\bx |\le M
\|\bu\|_{\bH^1(\Omega)}\|\bv\|_{\bH^1(\Omega)}$,
for all $\bu,\bv\in \bH^1_0(\Omega)$, where $u_1, \ldots, u_d$ are the
Cartesian components of $\bu$. This decomposition holds for instance
when
$\polA\bu = -\DIV(\mu(\GRAD \bu + (\GRAD \bu)\tr) + \kappa \DIV u
\polI)$
where $\polI$ is the $d\CROSS d$ identity matrix. Assuming in this
case that $\mu$ is constant over $\Omega$, we have
$A\bu = -\DIV (\mu \GRAD \bu)$ and $\lambda = \mu +\kappa$.

The first-order algorithm \eqref{first_order_art_comp_p_eliminated}
can be rewritten as follows in this new context:
\begin{equation}
\frac{\bu^{n+1}-\bu^n}{\dt} + A\bu^{n+1} -\GRAD (\varpi \DIV \bu^{n+1})
= \bef^{n+1} - \GRAD p^{n} ,
\quad p^{n+1} = -p^n  -\varpi\DIV\bu^{n+1}, \label{abtract_artificial_comp}
\end{equation}
where $\varpi := \lambda + \chi$ and we recall that $\chi = \tau/\epsilon$.

\subsection{Two-dimensional problems}\label{sec:2d}
Let us denote by $u_1, u_2$ the Cartesian components of $\bu$, \ie
$\bu=(u_1,u_2)\tr$.  We revisit the algorithm
\eqref{abtract_artificial_comp} and propose to consider the
following decoupled version thereof
\begin{align}
\frac{\bu^{n+1}-\bu^n}{\dt} + A\bu^{n+1} -
\left(\begin{aligned}
&\dx (\varpi(\dx u_1^{n+1} + \dy u_2^{n\phantom{+1}}  ))\\[-5pt]
&\dy (\varpi(\dx u_1^{n+1} + \dy u_2^{n+1} ))
\end{aligned}\right)
= \bef^{n+1} - \GRAD p^{n}
\label{abtract_artificial_comp_2D}
\end{align}
with $p^{n+1} = p^n  -\varpi\DIV\bu^{n+1}$. 
 Note that since we assumed
that $A$ is block diagonal, meaning that $A\bu = (A_1 u_1, A_2 u_2)$,
the Cartesian components of $\bu$ are indeed decoupled because the
algorithm can be recast as follows:
\begin{equation}
 \begin{cases}
   & \displaystyle \frac{u_1^{n+1}-u_1^n}{\dt} + A_1 u_1^{n+1}
   - \dx(\varpi\dx u_1^{n+1} )
   =f_1^{n+1} - \dx
   \left(p^{n}-\varpi\dy u_2^{n} \right)\\
   & \displaystyle \frac{u_2^{n+1}-u_2^n}{\dt} + A_2 u_2^{n+1} -
   \dy(\varpi \dy u_2^{n+1} )
   =f_2^{n+1}-\dy \left(p^{n}-\varpi\dx  u_1^{n+1} \right).
 \end{cases}
 \label{eq:gs_grad_div2}
\end{equation}
These two problems only require to solve classical scalar-valued
parabolic equations.  Before going through the stability analysis, let
us first observe that
\eqref{abtract_artificial_comp_2D} can be rewritten as
follows:
\begin{align}
\frac{\bu^{n+1}-\bu^n}{\dt} + A\bu^{n+1} -\GRAD(\varpi\DIV\tbu^{n+1} )
- (0,\dy (\varpi\dy \delta_t u_2^{n+1}))\tr
= \bef^{n+1} - \GRAD p^{n}
\label{abtract_artificial_comp_2D_bis}
\end{align}
where $\tbu^{n+1} = (u_1^{n+1},u_2^n)\tr$.  We assume that
$\bef=(f_1,f_2)=0$ in order to establish the stability of the scheme
with respect to the initial data.  The case of a non-zero source term
can be considered similarly, but since this unnecessarily introduces
irrelevant technicalities we will omit the source term in the rest of
the paper.  The scheme \eqref{abtract_artificial_comp_2D} is
unconditionally stable as stated by the following theorem.
\begin{theorem}\label{Thm:gs_grad_div_stability}
  Under suitable initialization and smoothness assumptions, the
  algorithm \eqref{abtract_artificial_comp_2D} is
  unconditionally stable, \ie for any finite time interval $(0, T]$ we
  have:
\begin{multline}
  \|\bu_{\dt}\|^2_{\ell^\infty(\Ldeuxd)} + \dt
  \|\varpi^{-\frac12}p_{\dt}\|_{\ell^\infty(\Ldeux)}^2 + \dt\|\varpi^{\frac12}\dy
  u_{2,{\dt}}\|_{\ell^\infty(\Ldeux)}^2
  + \dt^{-1}\|\delta_t \bu_{\dt}\|^2_{\ell^2(\Ldeuxd)} \\
  + 2\nu \|\bu_{\dt}\|^2_{\ell^2(\bH^1(\Omega))} +
  \|\varpi^{\frac12}\DIV\tbu\|_{\ell^2(L^2(\Omega))}^2 \le
  c(\|\bu^0\|^2_{\Ldeuxd} + \dt\|\varpi^{-\frac12}p^0\|^2_{L^2(\Omega)} +
  \dt \|\varpi^{\frac12}\dy u_2^0\|^2_{\Ldeuxd}).
\end{multline}
\end{theorem}
\begin{proof}
  We first multiply the momentum equation in
  \eqref{abtract_artificial_comp_2D_bis} by $2 \dt
  \bu^{n+1}$, then, using the identity
  $2(a-b,a)=\|a\|^2+\|a-b\|^2-\|b\|^2$ and the coerciveness of $A$ in
  $\bH^1(\Omega)$, we obtain:
\begin{multline*}
  \|\bu^{n+1}\|^2_{\Ldeuxd}  +  \|\delta_t \bu^{n+1}\|^2_{\Ldeuxd} 
- \|\bu^{n}\|^2_{\Ldeuxd} +2 \nu \dt \| \bu^{n+1}\|^2_{\bH^1(\Omega)} 
+ 2\dt(\varpi\DIV \tbu^{n+1},\DIV \bu^{n+1})\\
  + \dt (\|\varpi^{\frac12}\dy u_2^{n+1}\|_{\Ldeuxd}^2 + \|\varpi^{\frac12}\dy
  \delta_t u_2^{n+1}\|_{\Ldeuxd}^2 - \|\varpi^{\frac12}\dy u_2^{n}\|_{\Ldeuxd}^2) -2
  \dt (p^{n},\DIV \bu^{n+1}) \leq 0.
\end{multline*}
Now taking the square
of the pressure equation $\varpi^{-\frac12}p^{n+1} = \varpi^{-\frac12}p^n - \varpi^{\frac12} \DIV \bu^{n+1}$ gives
\[
\dt\|\varpi^{-\frac12}p^{n+1}\|_{L^2(\Omega)}^2 = \dt\|\varpi^{-\frac12}p^{n}\|_{L^2(\Omega)}^2) 
- 2\dt (\DIV \bu^{n+1}, p^{n}) + \dt\|\varpi^{\frac12}\DIV \bu^{n+1}\|_{\Ldeux}^2.
\]
Adding the above inequality and equation, we obtain:
\begin{multline*}
  \|\bu^{n+1}\|^2_{\Ldeuxd}  +  \|\delta_t \bu^{n+1}\|^2_{\Ldeuxd} 
 +2 \nu \dt \| \bu^{n+1}\|^2_{\bH^1(\Omega)} + \dt \|\varpi^{-\frac12}p^{n+1}\|_{L^2(\Omega)}^2 \\
+ \dt\|\varpi^{\frac12}\DIV \tbu^{n+1}\|_{\Ldeux}^2 
- \dt\|\varpi^{\frac12}\DIV (\bu^{n+1} - \tbu^{n+1})\|_{\Ldeux}^2 
+ \dt\|\varpi^{\frac12}\dy \delta_t u_2^{n+1}\|_{L^2(\Omega)}^2 \\
\dt\|\varpi^{\frac12}\dy u_2^{n+1}\|_{L^2(\Omega)}^2  
\leq  \|\bu^{n}\|^2_{\Ldeuxd} + \dt \|\varpi^{-\frac12}p^{n}\|_{L^2(\Omega)}^2
+ \dt\|\varpi^{\frac12}\dy u_2^n\|_{L^2(\Omega)}^2.
\end{multline*}
Note that $\DIV (\bu^{n+1} - \tbu^{n+1}) = \dy \delta_t u_2^{n+1}$, \ie
$\|\varpi^{\frac12}\DIV (\bu^{n+1} - \tbu^{n+1})\|_{\Ldeux} =
\|\varpi^{\frac12}\dy \delta_t u_2^{n+1}\|_{L^2(\Omega)}$.
Then summing the above inequality for $n=0, \dots, N-1$, with
$N=\floor{T/\dt}$, yields the desired result.
\end{proof}

The algorithm \eqref{abtract_artificial_comp_2D} can
be thought of as a Gauss-Seidel approximation of
\eqref{abtract_artificial_comp}. This observation, then
leads us to think of using the Jacobi approximation which consists of
replacing $\GRAD\DIV \bu^{n+1}$ in
\eqref{abtract_artificial_comp} by $\GRAD(\varpi\DIV \bu^n) + (\dx(\varpi\dx \delta_t
u_1^{n+1}),\dy(\varpi\dy \delta_t u_2^{n+1}))\tr$, that is to say
\begin{align}
\frac{\bu^{n+1}-\bu^n}{\dt} + A\bu^{n+1} - \GRAD(\varpi\DIV \bu^n)
-\left(\begin{matrix}
\dx (\varpi\dx \delta_t u_1^{n+1})\\
\dy (\varpi\dy \delta_t u_2^{n+1})
\end{matrix}\right)
= \bef^{n+1} - \GRAD p^{n}
\label{abtract_artificial_comp_Jacobi_2D}
\end{align}
with $p^{n+1} = p^n -\varpi \DIV \bu^{n+1}$. Let us define $\cbu^{n+1}
= (u_1^n,u_2^{n+1})\tr$.
\begin{theorem}\label{Thm:Jacobi_2D}
  Under suitable initialization and smoothness assumptions, the
  Jacobi algorithm \eqref{abtract_artificial_comp_Jacobi_2D} is
  unconditionally stable, \ie for any finite time interval $(0, T]$ we
  have:
\begin{multline}
  \|\bu_{\dt}\|^2_{\ell^\infty(\Ldeuxd)} + \dt
  \|\varpi^{-\frac12}p_{\dt}\|_{\ell^\infty(\Ldeux)}^2 
+ \dt\|\varpi^{\frac12}\dx u_{1,{\dt}}\|_{\ell^\infty(\Ldeux)}^2
+ \dt\|\varpi^{\frac12}\dy u_{2,{\dt}}\|_{\ell^\infty(\Ldeux)}^2\\
  + \dt^{-1}\|\delta_t \bu_{\dt}\|^2_{\ell^2(\Ldeuxd)} 
  + 2\nu \|\bu_{\dt}\|^2_{\ell^2(\bH^1(\Omega))} 
+ \|\varpi^{\frac12}\DIV \tbu\|_{\ell^2(L^2(\Omega))}^2 +
\|\varpi^{\frac12}\DIV \cbu\|_{\ell^2(L^2(\Omega))}^2 \\
\le
  c(\|\bu^0\|^2_{\Ldeuxd} +
  \dt\|\varpi^{-\frac12}p^0\|^2_{L^2(\Omega)} 
+ \dt \|\varpi^{\frac12}\dx u_1^0\|^2_{\Ldeuxd}
+ \dt \|\varpi^{\frac12}\dy u_2^0\|^2_{\Ldeuxd}).
\end{multline}
\end{theorem}

\subsection{Jacobi ansatz in higher dimensions}
More generally in $d$ dimension one could think of replacing
$\GRAD\DIV \bu^{n+1}$ by
$\GRAD\DIV \bu^n + (\dxx \delta_t u_1^{n+1},\ldots, \partial_{x_d x_d}
\delta_t u_d^{n+1})\tr$.
This approximation may be stable in dimension three
but we did not make attempts to verify this.
However, the following alternative perturbation is also first-order
consistent
$\GRAD\DIV \bu^n + d (\dxx \delta_t u_1^{n+1},\ldots, \partial_{x_d
  x_d} \delta_t u_d^{n+1})\tr$, and we can consider the algorithm
\begin{align}
\frac{\bu^{n+1}-\bu^n}{\dt} + A\bu^{n+1} - \GRAD(\varpi\DIV \bu^n)
-d\left(\begin{matrix}
\dx (\varpi\dx \delta_t u_1^{n+1})\\\ldots\\
\dxd (\varpi\dxd \delta_t u_d^{n+1})
\end{matrix}\right)
= \bef^{n+1} - \GRAD p^{n}
\label{abtract_artificial_comp_Jacobi_nD}
\end{align}
with $p^{n+1} = p^n -\varpi \DIV \bu^{n+1}$. 
\begin{theorem}\label{Thm:Jacobi_nD}
  Under suitable initialization and smoothness assumptions, the
  Jacobi algorithm \eqref{abtract_artificial_comp_Jacobi_nD} is
  unconditionally stable, \ie for any finite time interval $(0, T]$ we
  have:
\begin{multline}
  \|\bu_{\dt}\|^2_{\ell^\infty(\Ldeuxd)} + \dt
  \|\varpi^{-\frac12}p_{\dt}\|_{\ell^\infty(\Ldeux)}^2 
+ \dt d\sum_{i=1}^d\|\varpi^{\frac12}\dxi u_{i,{\dt}}\|_{\ell^\infty(\Ldeux)}^2\\
  + \dt^{-1}\|\delta_t \bu_{\dt}\|^2_{\ell^2(\Ldeuxd)} 
  + 2\nu \|\bu_{\dt}\|^2_{\ell^2(\bH^1(\Omega))} 
+ \|\varpi^{\frac12}\DIV \bu\|_{\ell^2(L^2(\Omega))}^2 \\
\le
  c(\|\bu^0\|^2_{\Ldeuxd} +
  \dt\|\varpi^{-\frac12}p^0\|^2_{L^2(\Omega)} 
+ \dt d\sum_{i=1}^d\|\varpi^{\frac12}\dxi u_i^0\|_{\Ldeux}^2
\end{multline}
\end{theorem}
\begin{proof} Proceeding as in the proof of
  Theorem~\ref{Thm:gs_grad_div_stability}, we obtain
\begin{multline*}
  \|\bu^{n+1}\|^2_{\Ldeuxd}  +  \|\delta_t \bu^{n+1}\|^2_{\Ldeuxd} 
 +2 \nu \dt \| \bu^{n+1}\|^2_{\bH^1(\Omega)} + \dt \|\varpi^{-\frac12}p^{n+1}\|_{L^2(\Omega)}^2
- \dt\|\varpi^{\frac12}\DIV \bu^{n+1}\|_{\Ldeux}^2 \\
+ 2\dt(\varpi\DIV \bu^n,\DIV\bu^{n+1}) 
+\dt d\sum_{i=1}^d \big(\|\varpi^{\frac12}\dxi u_i^{n+1}\|_{\Ldeux}^2 +
\|\varpi^{\frac12}\dxi \delta_t u_i^{n+1}\|_{\Ldeux}^2\big)\\
\leq  \|\bu^{n}\|^2_{\Ldeuxd} + \dt \|\varpi^{-\frac12}p^{n}\|_{L^2(\Omega)}^2
+ \dt d\sum_{i=1}^d \|\varpi^{\frac12}\dxi u_i^{n}\|_{\Ldeux}^2.
\end{multline*}
We now observe that
$(\DIV\delta_t \bu^{n+1})^2 \le d\sum_{i=1}^d(\dxi\delta_t u_i)^2$,
which in turn implies that
\begin{multline*}
-\|\varpi^{\frac12}\DIV \bu^{n+1}\|_{\Ldeux}^2 + 2(\varpi\DIV
\bu^n,\DIV\bu^{n+1}) + d\sum_{i=1}^d\|\varpi\dxi\delta_t u_i\|_{\Ldeux}^2\\
= -\|\varpi^{\frac12}\DIV \delta_t \bu^{n+1}\|_{\Ldeux}^2 +
\|\varpi^{\frac12}\DIV \bu^n\|_{\Ldeux}^2 
+ d\sum_{i=1}^d\|\varpi\dxi\delta_t u_i\|_{\Ldeux}^2\ge \|\varpi^{\frac12}\DIV \bu^n\|_{\Ldeux}^2.
\end{multline*}
The conclusion follows readily.
\end{proof}

\subsection{Three-dimensional problems}\label{sec:3d}
The Gauss-Seidel scheme introduced in the previous section can be directly extended to the
three dimensional case:
\begin{align}
\frac{\bu^{n+1}-\bu^n}{\dt} + A\bu^{n+1} -
\left(\begin{aligned}
&\dx (\varpi(\dx u_1^{n+1} + \dy u_2^{n\phantom{+1}} + \dz u_3^{n\phantom{+1}} ))\\[-5pt]
&\dy (\varpi(\dx u_1^{n+1} + \dy u_2^{n+1}+ \dz u_3^{n\phantom{+1}} ))\\[-5pt]
&\dy (\varpi(\dx u_1^{n+1} + \dy u_2^{n+1}+ \dz u_3^{n+1} ))
\end{aligned}\right)
= \bef^{n+1} - \GRAD p^{n}
\label{abtract_artificial_comp_3D}
\end{align}
with $p^{n+1} = p^n  -\varpi\DIV\bu^{n+1}$. Then again the three
Cartesian components of the velocity are decoupled. Unfortunately, we
have not been able to prove the stability of this scheme, but our
numerical experiments lead us to conjecture that it is unconditionally stable. 
We have found though that stability can be proved by adding the
first-order perturbation 
$-(0,\dy(\varpi\dy\delta_tu_2^{n+1}),\dz(\varpi(-\dy\delta_t
u_2^{n+1}+\dz\delta_tu_3^{n+1})))\tr$, leading to the following scheme
\begin{align}
\frac{\bu^{n+1}-\bu^n}{\dt} + A\bu^{n+1} -
\left(\begin{aligned}
&\dx (\varpi(\dx u_1^{n+1} + \dy u_2^{n\phantom{+1}} + \dz u_3^{n} ))\\[-5pt]
&\dy (\varpi(\dx u_1^{n+1} + \dy (2u_2^{n+1}-u_2^n)+ \dz u_3^{n} ))\\[-5pt]
&\dy (\varpi(\dx u_1^{n+1} + \dy u_2^{n}+ \dz (2u_3^{n+1} -u_3^n)))
\end{aligned}\right)
= \bef^{n+1} - \GRAD p^{n}.
\label{abtract_artificial_comp_modif_3D}
\end{align}
Stability will be established by relying on the following result.
\begin{lemma} \label{Lem:a1_b1_c1}
 Let $a_1, b_1, c_1, b_0, c_0$ three real numbers, then
the following identity holds:
\begin{multline}
2((a_1+b_0+c_0)a_1 + (a_1+b_1+c_0)b_1 + (a_1+b_1+c_1)c_1) \\
+2(b_1-b_0)b_1 - 2(b_1-b_0)c_1 + 2(c_1-c_0)c_1 
=(a_1+b_1+c_1)^2 \\+ (a_1+b_0+c_0)^2 
+ 2(b_1^2+c_1^2-b_0^2-c_0^2)
+(b_1-b_0-c_1+c_0)^2.
\end{multline}
\end{lemma}

\begin{theorem} \label{Thm:GS_modiied_3D} Under suitable
  initialization and smoothness assumptions (assuming that $\bef=0$),
  the algorithm \eqref{abtract_artificial_comp_modif_3D} is
  unconditionally stable, \ie upon setting
  $\tbu^{n+1}=(u_1^{n+1},u_2^n,u_3^n)\tr$, the following holds for any
  finite time interval $(0, T]$:
\begin{multline}
 \|\bu_{\dt}\|^2_{\ell^\infty(\Ldeuxd)}  
+ \dt\|\varpi^{-\frac12}p_{\dt}\|^2_{\ell^\infty(L^2(\Omega))} 
+2 \dt|\varpi^{\frac12}\dy u_{2,\dt}\|^2_{\ell^\infty(L^2(\Omega))} \\
+2 \dt\|\varpi^{\frac12}\dz  u_{3,\dt}\|^2_{\ell^\infty(L^2(\Omega))}
+\dt^{-1} \|\delta_t
 \bu_{\dt}\|^2_{\ell^2(\Ldeuxd)}   +2 \nu
 \|\bu_{\dt}\|^2_{\ell^2(\bH^1(\Omega))}  \\
  + \|\varpi^{\frac12}\DIV \tbu_\dt\|^2_{\ell^2(L^2(\Omega))}  + \|\varpi^\frac12(\dy
  \delta_t u_{2,\dt}-\dz \delta_t u_{3\dt})\|^2_{\ell^2(L^2(\Omega))}  \\ 
\le c \left( \|\bu^{0}\|^2_{\Ldeuxd}+\|\varpi^{-\frac12}p^{0}\|^2_{L^2(\Omega)} +
  2\dt \|\varpi^{\frac12}\dy u_2^{0}\|^2_{L^2(\Omega)}+2\dt \|\varpi^{\frac12}\dz u_3^{0}\|^2_{L^2(\Omega)} \right).
\label{pr:3d_3}
\end{multline}
\end{theorem}
\begin{proof}
  The stability is be established by proceeding as in the two
  dimensional case.  Assuming that $\bef=0$, we first multiply the
  first three equations in \eqref{abtract_artificial_comp_modif_3D} by
  $2\dt\bu^{n+1}$, then using the identity
  $2(a-b,a)=\|a\|^2+\|a-b\|^2-\|b\|^2$ and Lemma~\ref{Lem:a1_b1_c1} to
  handle the $\GRAD\DIV$ term, we have
\begin{multline*}
  \|\bu^{n+1}\|^2_{\Ldeuxd}  +  \|\delta_t \bu^{n+1}\|^2_{\Ldeuxd} 
-\|\bu^{n}\|^2_{\Ldeuxd} 
+2 \nu \dt \| \bu^{n+1}\|^2_{\bH^1(\Omega)} 
-2 \dt (p^{n},\DIV \bu^{n+1})\\
+ \dt\|\varpi^{\frac12}\DIV \bu^{n+1}\|_{\Ldeux}^2 + \dt\|\varpi^{\frac12}\DIV \tbu^{n+1}\|_{\Ldeux}^2
+2\dt\|\varpi^{\frac12}\dy u_2^{n+1}\|^2_{L^2(\Omega)}+2\dt\|\dz u_3^{n+1}\|^2_{L^2(\Omega)} \\
-2\dt\|\varpi^{\frac12}\dy u_2^{n}\|^2_{L^2(\Omega)}
-2\dt\|\varpi^{\frac12}\dz u_3^{n}\|^2_{L^2(\Omega)}
+\dt\|\varpi^{\frac12}(\dy\delta_t u_2^{n+1} - \dz\delta_t u_3^{n+1})\|_{\Ldeux}^2\le 0.
\end{multline*}
Then we add the pressure equation
\[
\dt\|\varpi^{-\frac12}p^{n+1}\|_{L^2(\Omega)}^2 = \dt\|\varpi^{-\frac12}p^{n}\|_{L^2(\Omega)}^2) 
- 2\dt (\DIV \bu^{n+1}, p^{n}) + \dt\|\varpi^{\frac12}\DIV \bu^{n+1}\|_{\Ldeux}^2,
\]
and obtain
\begin{multline*}
  \|\bu^{n+1}\|^2_{\Ldeuxd}  +  \|\delta_t \bu^{n+1}\|^2_{\Ldeuxd} 
+2 \nu \dt \| \bu^{n+1}\|^2_{\bH^1(\Omega)} 
+\dt\|\varpi^{-\frac12}p^{n+1}\|_{L^2(\Omega)}^2
+ \dt\|\varpi^{\frac12}\DIV \tbu^{n+1}\|_{\Ldeux}^2\\
+2\dt\|\varpi^{\frac12}\dy u_2^{n+1}\|^2_{L^2(\Omega)}+2\dt\|\dz u_3^{n+1}\|^2_{L^2(\Omega)} 
+ \dt\|\varpi^{\frac12}(\dy\delta_t u_2^{n+1} - \dz\delta_t u_3^{n+1})\|_{\Ldeux}^2\\
\le \|\bu^{n}\|^2_{\Ldeuxd} +
\dt\|\varpi^{-\frac12}p^{n}\|_{L^2(\Omega)}^2) 
+2\dt\|\varpi^{\frac12}\dy u_2^{n}\|^2_{L^2(\Omega)}
+2\dt\|\varpi^{\frac12}\dz u_3^{n}\|^2_{L^2(\Omega)}.
\end{multline*}
Finally, the result follows by summing the above inequality for
$n=0, \dots, N-1$.
\end{proof}

\section{Direction splitting schemes}
 Direction splitting algorithms based on the artificial compressibility
formulation of the Navier-Stokes equations have been proposed many
years ago (see \cite{MR46:6613}, section 8.3, \cite{Ladhyz70b},
chapter VI, section 9.2, \cite{Te77}, chapter III, section 8.3), and
they have 
largely been abandoned in the last twenty years.  Restricting the
discussion to two dimensions for simplicity,
all of the above direction splitting schemes can be
considered as discretizations of the following set of PDEs formulated
in \citep{MR46:6613} and approximating the incompressible Navier-Stokes
equations with constant viscosity:
\begin{equation}
 \left\{\begin{aligned}
 &\frac{1}{2} \partial_t u_1 + u_1\dx u_1 + \dx  p = \nu \dxx u_1,\\
 &\frac{1}{2} \partial_t u_2 + u_1\dx u_2  = \nu \dxx u_2,\\
 &\frac{\epsilon}{2} \partial_t p + \epsilon u_1 \dx p + p \dx u_1= 0,\\
  \end{aligned}\right.
\label{eq:Yan1}
\end{equation} 
in the first half of a given time interval $[t^n, t^n+\frac12 \dt]$ and 
\begin{equation}
 \left\{\begin{aligned}
 &\frac{1}{2} \partial_tu_1 + u_2\dy u_1 = \nu \dyy u_1,\\
 &\frac{1}{2} \partial_tu_2 + u_2\dy u_2 + \dy p = \nu \dyy u_2,\\
 &\frac{\epsilon}{2}  \partial_t p + \epsilon u_2 \dy p + p \dy u_2= 0,\\
  \end{aligned} \right.
\label{eq:Yan2}
\end{equation} 
in the second half $[t^n+\frac12\dt, t^{n+1}]$.  Note that in  \cite{Ladhyz70b} and  \cite{Te77} the pressure
equations are formulated slightly differently:
\begin{align}
 \frac{\epsilon}{2}  \partial_t p + \dx u_1&= 0, \quad \text{ in }
  [t^n, t^n+\frac12\dt], \\
 \frac{\epsilon}{2} \partial_t p + \dy u_2&= 0,\quad \text{ in } [t^n+\frac12\dt, t^{n+1}],
 \end{align} 
 In the scheme of \cite{MR46:6613} the pressure equations are derived
 from the compressible mass conservation equation at vanishing Mach
 number, in \citep{Ladhyz70b} and \citep{Te77} they are derived from
 the simpler (but less physical) perturbation of the incompressibility
 constraint: $\epsilon \partial_t p + \DIV \bu=0$.  Both algorithms
 are formally first order accurate in time. However, the actual rate
 of convergence was not established in the above references, despite that convergence was proven
 in both cases.
 
 In the present paper we are aiming at the development of artificial
 compressibility schemes of order two and higher.  In \cite{GM14} we
 proposed two possible approaches for extending the convergence order.
 The first one uses a bootstrapping perturbation of the
 incompressibility constraint combined with a high order BDF time
 stepping for the momentum equation.  The second approach is based on
 a defect (or deferred) correction for both, the momentum and the
 continuity equations.  Since we are presently unable to devise a
 higher order defect correction scheme based on any of the first order
 direction splitting methods discussed above (see \eg
 \eqref{eq:Yan1}-\eqref{eq:Yan2}), we consider here a scheme that is a
 first order perturbation of the formally second order splitting
 scheme due to \cite{D62}.  For simplicity, we will not consider the
 nonlinear terms in what follows, however, there is no particular
 difficulty to extend the scheme to the nonlinear case by using Euler
 explicit discretization. It is also possible to discretize the
 nonlinear terms semi-implicitly by proceeding as in
 \citep{MR46:6613}, \citep{Ladhyz70b}, and \citep{Te77}.  Denoting by
 $p^{n+1/2}$ and $\bar{p}_{\dt}$ the approximation of the pressure at
 time $t^{n+1/2}$ and the time sequence of pressure values,
 respectively, the derivation of the scheme starts from the
 Crank-Nicolson discretization of the momentum equation of the
 artificial compressibility system that is given by:
\begin{equation}
 \left\{\begin{aligned}
 &\frac{\delta_t \bu^{n+1}}{\dt} + \frac{1}{2} A (\bu^{n+1}+\bu^n)=- \frac{1}{2} \GRAD p^{n+1/2}  + \bef^{n+1/2},\\
 & (p^{n+1/2} -p^{n-1/2}) + \frac{\varpi }{2}\DIV (\bu^{n+1}+\bu^n) = 0,
  \end{aligned}\right.
  \label{eq:ac_cn}
\end{equation}

Let us assume that the operators $A_1$ and $A_2$ can be split into a sum of
two self-adjoint semi-definite positive operators i.e., $A_1=A_{11}+A_{12}$ and $A_2=A_{21}+A_{22}$.
For example, if $A_1 = -\nu \LAP$ and $A_2= -\nu \LAP$ then the direction
splitting algorithm presumes the splitting  $A_{11}=-\nu \dxx$, $A_{12}=-\nu \dyy$, 
$A_{21}=-\nu \dxx$, $A_{22}=-\nu \dyy$.  Let us also assume that $\varpi$ is
constant over $\Omega$ and introduce the operators: $C_{11}=-\varpi \dxx$, $C_{12}=-\varpi \dxy$,
$C_{21}=-\varpi \dyx$, $C_{22}=-\varpi \dyy$. 
Then the direction splitting scheme is given by:
\begin{equation}
 \left\{\begin{aligned}
 & \left(I + \frac{\dt}{2}
   \left(A_{11}+C_{11}\right)\right)\left(I+\frac{\dt}{2} A_{12}\right)
 \frac{\delta_t u_1^{n+1}}{\dt} 
+ \left(A_{11} + C_{11}+A_{12}\right) u_1^{n}=\\
&  -\frac{1}{2}C_{12} (u_2^{n}+u_2^{n-1}) - \dx p^{n-1/2} + f_1^{n+1/2},\\
 & \displaystyle \left(I + \frac{\dt}{2}
   \left(A_{22}+C_{22}\right)\right)\left(I+\frac{\dt}{2} A_{21}\right)
 \frac{\delta_t u_2^{n+1}}{\dt} 
+ \left(A_{22} + C_{22}+A_{21}\right) u_2^{n}= \\
& -\frac{1}{2} C_{21}  (u_1^{n+1}+u_1^n) - \dy p^{n-1/2} + f_2^{n+1/2}, \\
 & (p^{n+1/2} -p^{n-1/2}) + \frac{\varpi }{2}\DIV\ (\bu^{n+1}+\bu^n) = 0,
  \end{aligned}\right.
\label{eq:split_1}
\end{equation}
where $I$ is the identity operator. Note that this is a perturbation
of the Crank-Nicolson discretization \eqref{eq:ac_cn} that includes
the formally second order terms
$\dt (A_{11}+C_{11}) A_{12} \delta_t u_1^{n+1}/4, \dt (A_{22}+C_{22})
A_{21} \delta_t u_2^{n+1}/4$,
and the term $C_{12} (u_2^{n+1}+u_2^{n}) $ is extrapolated by
$C_{12} (u_2^{n}+u_2^{n-1})$.  The last perturbation is first order
accurate of course, but since the perturbation of the
incompressibility constraint is also first order, it does not change
the overall first order approximation of the unsteady Stokes
equations.  In the next section we will demonstrate how to correct
these first order defects of the scheme and lift the accuracy to
second order.

 Let us now assume that the operator $A_{11} + C_{11}$ commutes with $ A_{12}$, and  $A_{22}+ C_{22}$ 
commutes with $A_{21}$. Such commutativity conditions are satisfied if, for 
example, the viscosity $\nu$ is constant and if the domain boundary consists 
of straight lines parallel to one of the coordinate axes.  Then the operator
\begin{align*}
&\bB = (B_1, B_2)^T=(A_{11}+ C_{11})A_{12}, (A_{22}+ C_{22})A_{21})^T=\\
&\frac{1}{2} ((A_{11}+C_{11})A_{12}+A_{12}(A_{11}+ C_{11}), (A_{22}+ C_{22})A_{21}+A_{21}(A_{22}+ C_{22}))^T
\end{align*}
is a self-adjoint positive semi-definite operator defining a semi-norm that we denote by
$|.|_{\bB(\Omega)}$. Under such conditions it is quite straightforward to prove the following theorem providing the stability  estimate for the splitting scheme.
The stability without the commutativity assumption is significantly more difficult to verify, particularly in 3D, and it is still an open problem (see for example
the discussion about splitting schemes for non-commutative operators in \cite{Vab14}).
\begin{theorem}\label{Thm:split_gs_grad_div_stability}
  Under suitable initialization and smoothness assumptions, if 
  $(A_{11}+C_{11})A_{12}=A_{12}(A_{11}+C_{11}), A_{22}+C_{22})A_{21}=A_{21}(A_{22}+C_{22})$, 
  and if $\bef=0$, the algorithm \eqref{eq:split_1} is unconditionally stable, \ie for any finite time interval $(0, T]$ we have:
\begin{align*}
 & \|\bu_{\dt}\|^2_{\ell^\infty(\Ldeuxd)} + \dt
  \|\varpi^{-\frac12}\bar{p}_{\dt}\|_{\ell^\infty(\Ldeux)}^2 + \dt\|\varpi^{\frac12}\dy
  \bar{u}_{2,{\dt}}\|_{\ell^\infty(\Ldeux)}^2 +\\ 
  &+ 2\nu \|\bar{\bu}_{\dt}\|^2_{\ell^2(\bH^1(\Omega))} +  \frac{\dt^2}{4}|{\bu}_{\dt}|^2_{\ell^\infty(\bB(\Omega))} +
  \|\varpi^{\frac12}\DIV\tbu\|_{\ell^2(L^2(\Omega))}^2 \le\\
  &c(\|\bu^0\|^2_{\Ldeuxd} + \dt\|\varpi^{-\frac12}p^{-1/2}\|^2_{L^2(\Omega)} +
  \dt \|\varpi^{\frac12}\dy \bar{u}_2^0\|^2_{\Ldeuxd}+ \frac{\dt^2}{4}|{\bu}^0|^2_{\bB(\Omega)}),
\end{align*}
where $\bar{u}_2^0=u_2^0$, ${p}^{-1/2}=p^0$, and $\tbu^{n+1}=(\bar{u}_1^{n+1},\bar{u}_2^n)\tr$.
\end{theorem}
\begin{proof}
We first notice that the momentum equation in \eqref{eq:split_1} can be rewritten in a form similar to \eqref{abtract_artificial_comp_2D_bis}:
\begin{align*}
\frac{\bu^{n+1}-\bu^n}{\dt} + A\bar{\bu}^{n+1}+\frac{\dt^2}{4}B\frac{\bu^{n+1}-\bu^n}{\dt} -
\GRAD(\varpi\DIV\tbu^{n+1} ) - (0,\dy (\varpi\dy \delta_t \bar{u}_2^{n+1}))\tr = - \GRAD p^{n-1/2}
\end{align*}
where $\tbu^{n+1}=(\bar{u}_1^{n+1},\bar{u}_2^n)\tr$.  Multiplying this equation by $2 \dt
 \bar{\bu}^{n+1}$, then using the identities  $2(a-b)a=\|a\|^2+\|a-b\|^2-\|b\|^2$ and $(a-b)(a+b)=\|a\|^2-\|b\|^2$,
and the coerciveness of $A$ in $\bH^1(\Omega)$, we obtain:
\begin{align*}
 & \|\bu^{n+1}\|^2_{\Ldeuxd} - \|\bu^{n}\|^2_{\Ldeuxd} +2 \nu \dt \| \bar{\bu}^{n+1}\|^2_{\bH^1(\Omega)} 
+ 2\dt(\varpi\DIV \tbu^{n+1},\DIV \bar{\bu}^{n+1}) + 
 \frac{\dt^2}{4}\left( |{\bu}^{n+1}|^2_{\bB(\Omega)} - |{\bu}^n|^2_{\bB(\Omega)}\right)\\
 & + \dt \left(\|\varpi^{\frac12}\dy \bar{u}_2^{n+1}\|_{\Ldeuxd}^2 + 
  \|\varpi^{\frac12}\dy \delta_t \bar{u}_2^{n+1}\|_{\Ldeuxd}^2 - \|\varpi^{\frac12}\dy \bar{u}_2^{n}\|_{\Ldeuxd}^2\right) -2
  \dt (p^{n-1/2},\DIV \bar{\bu}^{n+1}) \leq 0.
\end{align*}
The rest of the proof follows along the same lines as the proof of theorem \eqref{Thm:gs_grad_div_stability}.
\end{proof}
%

\section{Higher order methods}
The first order schemes discussed in the previous  two sections can be extended to second order by  at least two possible approaches described in \cite{GM14}.  
The resulting schemes are quite efficient if the linear systems are solved by means of iterative solvers.
In order to handle the 2D and 3D case together it is convenient to introduce the following operator corresponding to the mixed second order
derivatives appearing in the formulation:
\begin{eqnarray*}
C_{\triangle}  =  \begin{bmatrix} 0 & C_{12} \\ 0 & 0 \end{bmatrix} \text{ in 2D,  and } \\
C_{\triangle} =   \begin{bmatrix} 0 & C_{12} & C_{13} \\ 0 & 0 & C_{23} \\0 & 0 & 0\end{bmatrix} \text{ in 3D,}
\end{eqnarray*} 
with $C_{13}, C_{23}$ being defined similarly to $C_{12}$ i.e. $C_{i3} = \varpi \partial_{x_i x_3 }, i=1,2$.
An example of a 2D second order BDF bootstrapping procedure based on \eqref{eq:gs_grad_div2} and analogous to the scheme (5.1)-(5.2) of  \cite{GM14} is given by:
\begin{equation}
 \begin{cases}
 & \displaystyle  \frac{ \tbu^{n+1}- \tbu^n}{\dt}  + A \tbu^{n+1} + \GRAD {\tp}_1^{n+1} -  C_{\triangle} (\tbu^{n+1}-\tbu^{n})=\bef^{n+1},\\
  & \displaystyle   {\tp}^{n+1} - {\tp}^n + \varpi \DIV \tbu^{n+1} = 0,\\
 & \displaystyle  \frac{3 {\bu}^{n+1}-4 {\bu}^n+{\bu}^{n-1}}{2\dt}  + A\bu^{n+1} + \GRAD p^{n+1} - C_{\triangle} ({\bu}^{n+1}-2{\bu}^{n}+\bu^{n-1})=\bef^{n+1},\\
 & \displaystyle    p^{n+1} - p^{n} - (\tp^{n+1} - \tp^n) + \varpi \DIV{\bu}^{n+1} = 0.\\
  \end{cases}
\label{eq:dc_grad_div}
\end{equation} 
Note that the only difference with the scheme (5.1)-(5.2) of
\cite{GM14} is the presence of the terms
$ C_{\triangle} (\tbu^{n+1}-\tbu^{n})$ and
$C_{\triangle} ({\bu}^{n+1}-2{\bu}^{n}+\bu^{n-1})$ in the two momentum
equations.  Presuming enough smoothness of the exact solution, these
terms are of order $\dt$ and $\dt^2$ respectively, and their
presence is compatible with the overall second order of consistency of the
scheme.  In the case of the Navier-Stokes equations the advection terms
can be approximated by means of a first and second order
Adams-Bashfort (AB2) schemes in the first and second stage of the
bootstrapping procedure in \eqref{eq:dc_grad_div}.

As shown in \cite{GM14}, in the case of the full Navier-Stokes equations,
the defect correction schemes have better stability properties than
the high order schemes based on BDF time stepping.  Using the
third order approximation to the velocity and pressure
$\bu_0^n+\dt \bu_1^n+\dt^2 \bu_2^n$, $p_0^n+\dt p_1^n+\dt^2 p_2^n$, we
can write the third order scheme with a decoupled {\em grad-div}
operator, analogous to the scheme
\eqref{defect_ns_bootstrap0}-\eqref{defect_ns_bootstrap2}, as:


\begin{equation} \text{for } n \ge 0, \qquad
\begin{cases}
  \textbf{nl}_0^{n+1}=B \bu_0^n, \\ \displaystyle
  \frac{\bu_0^{n+1}-\bu_0^n}{\dt} + A\bu_0^{n+1}  
-\varpi\GRAD\DIV \bu_0^{n+1} + C_{\triangle}(\bu_0^{n+1}-\bu_0^n)  + \GRAD p_0^{n}  = \bef^{n+1} - \textbf{nl}_0^{n+1}\\
  p_0^{n+1} = p_0^{n}  - \varpi \DIV\bu_0^{n+1},\\
  d\bu_0^{n+1} = (\bu_0^{n+1}-\bu_0^n)/\dt, \quad dp_0^{n+1} =
  (p_0^{n+1} -p_0^n)/\dt
\end{cases} \label{defect_ns_split0}
\end{equation}
\begin{equation} \text{for } n \ge 1, \qquad
\begin{cases}
d^2\bu_0^{n+1}= (d\bu_0^{n+1}-d\bu_0^{n})/\dt, \\
\textbf{nl}_1^n=B (\bu_0^{n} + \tau \bu_1^{n-1}), \\ \displaystyle
\frac{\bu_1^{n}-\bu_1^{n-1}}{\dt} + A\bu_1^{n}  -\varpi \GRAD\DIV
\bu_1^n + C_{\triangle}(\bu_1^{n}-\bu_1^{n-1}) - C_{\triangle}(\bu_0^{n}-\bu_0^{n-1}) +\\ \displaystyle
 \GRAD (p_1^{n-1} + dp_0^{n} ) = -\frac12 d^2\bu_0^{n+1}
- \frac{\textbf{nl}_1^n - \textbf{nl}_0^n}{\dt} + \varpi C d\bu_0^{n+1},\\
p_1^{n} = p_1^{n-1}  + dp_0^{n}  - \varpi  \DIV\bu_1^{n},\\
d\bu_1^{n} = (\bu_1^{n}-\bu_1^{n-1})/\dt, 
\quad dp_1^{n} = (p_1^{n}-p_1^{n-1})/\dt,
\end{cases}\label{defect_ns_split1}
\end{equation}
\begin{equation} \text{for } n \ge 2, \quad
\begin{cases}
d^2\bu_1^{n}=(d\bu_1^{n}-d\bu_1^{n-1})/\dt, \qquad
d^3\bu_0^{n+1}=(d^2\bu_0^{n+1}-d^2\bu_0^{n})/\dt, \\
\textbf{nl}_2^{n-1}=B (\bu_0^{n-1} + \tau \bu_1^{n-1} +\tau^2 \bu_2^{n-2})\\ \displaystyle
\frac{\bu_2^{n-1}-\bu_2^{n-2}}{\dt} + A\bu_2^{n-1} -\varpi\GRAD \DIV
\bu_2^{n-1} + C_{\triangle}(\bu_2^{n-1}-\bu_2^{n-2}) - C_{\triangle}(\bu_1^{n-1}-\bu_1^{n-2}) + \\ \displaystyle 
\GRAD (p_2^{n-2} + dp_1^{n-1})  =  -\frac12 d^2\bu_1^{n}  + \frac16 d^3\bu_0^{n+1}
 -  \frac{\textbf{nl}_2^{n-1} - \textbf{nl}_1^{n-1}}{\dt^2} + \varpi C d\bu_1^{n+1},\\
p_2^{n-1} = p_2^{n-2} + dp_1^{n-1} - \varpi \DIV\bu_2^{n-1},\\
\bu^{n-1} =  \bu_0^{n-1} + \tau \bu_1^{n-1} +\tau^2 \bu_2^{n-1}, \quad
p^{n-1} =  p_0^{n-1} + \tau p_1^{n-1} +\tau^2 p_2^{n-1}.
\end{cases}\label{defect_ns_split2}
\end{equation}
In 3D, this scheme is the defect correction extension of the scheme
\eqref{abtract_artificial_comp_3D}.  Although we are presently unable
to prove its stability, we use it in the numerical experiments
presented below. Our tests show that this scheme is unconditionally
stable in the case of the unsteady Stokes equations.

Note that all these schemes require only the solution of problems of the type 
\[
v- \dt \DIV (\kappa \GRAD v )=r,
\]
for each component of the velocity, where $\kappa$ is a diagonal
matrix.  For example, in 2D either
\[
\kappa = \left[\begin{matrix}
\displaystyle \nu+\varpi & 0 \\
 0&\displaystyle  \nu 
 \end{matrix}
  \right] \quad \text{or} \quad
\kappa = \left[\begin{matrix}
\displaystyle \nu& 0 \\
 0&\displaystyle  \nu +\varpi
 \end{matrix}
  \right],
 \]
 when we solve for the first or the second Cartesian component of
 the velocity, respectively.  The solution process for the
 incompressible unsteady Navier-Stokes equations is thereby reduced to
 the solution of a fixed number of classical parabolic problems.

Using the defect correction approach of \cite{GM14} the direction splitting scheme \eqref{eq:split_1} can also be extended to second order as follows:
\begin{equation}
 \begin{cases}
  &\displaystyle \left(I + \frac{\dt}{2} (A_{11}+C_{11})\right)\left(I+\frac{\dt}{2} A_{12}\right) \frac{\tu_1^{n+1} - \tu_1^{n}}{\dt} + (A_{11} + C_{11}+A_{12}) \tu_1^{n}=\\
  &\displaystyle  \hspace{5cm}-\frac{1}{2}C_{12} (\tu_2^{n}+\tu_2^{n-1}) - \dx \tp^{n-1/2} + f_1^{n+1/2},\\
  &\displaystyle \left(I + \frac{\dt}{2} (A_{22}+C_{22})\right)\left(I+\frac{\dt}{2} A_{21}\right) \frac{\tu_2^{n+1} - \tu_2^{n}}{\dt} + (A_{22} + C_{22}+A_{21}) \tu_2^{n}= \\
  & \displaystyle  \hspace{5cm}-\frac{1}{2}C_{21} ( \tu_1^{n+1} +\tu_1^{n})- \dy \tp^{n-1/2} + f_2^{n+1/2}, \\
  &\displaystyle \tp^{n+1/2} - \tp^{n-1/2} + \frac{\varpi}{2} \DIV  ( \tilde{\bu}^{n+1}+\tilde{\bu}^n) = 0,
  \end{cases}
  \label{eq:split_02}
\end{equation}
  \begin{equation}
  \begin{cases}
  & d \tu^{n+1} = (\tu_2^{n+1} -\tu_2^{n})/\dt, \\
 & \displaystyle \left(I + \frac{\dt}{2} (A_{11}+C_{11})\right)\left(I+\frac{\dt}{2} A_{12}\right) \frac{u_1^{n+1} - u_1^{n}}{\dt} + (A _{11}+ C_{11}+A_{12})) u_1^{n}=  \\
 &\displaystyle \hspace{4cm}-C_{12} \left(\frac{1}{2} (u_2^{n}+u_2^{n-1}) +\dt d \tu^{n+1} \right)- \dx p^{n-1/2} + f_1^{n+1/2},\\
 & \displaystyle \left(I + \frac{\dt}{2} (A_{22}+C_{22})\right)\left(I+\frac{\dt}{2} A_{21}\right) \frac{u_2^{n+1} - u_2^{n}}{\dt} + (A_{22} + C_{22}+A_{21}) u_2^{n}= \\
 &\displaystyle \hspace{4cm}-\frac{1}{2}C_{21} (u_1^{n+1} +u_1^{n})- \dy p^{n-1/2} + f_2^{n+1/2}, \\
 & \displaystyle p^{n+1/2} - p^{n-1/2} -(\tp^{n+1/2} - \tp^{n-1/2})+ \frac{\varpi}{2} \DIV (\bu^{n+1}+\bu^n) = 0.
  \end{cases}
\label{eq:split_12}
\end{equation}
The nonlinear terms of the Navier-Stokes equations can be included in
the above algorithm exactly as in the scheme \eqref{defect_ns_split0}-\eqref{defect_ns_split1}.

\section{Numerical results}
We first present some two dimensional numerical results comparing the
performance of the third order artificial compressibility method in
\cite{GM14}, \eqref{defect_ns_bootstrap0}-\eqref{defect_ns_bootstrap2}
and the scheme with the explicit mixed derivatives
\eqref{defect_ns_split0}-\eqref{defect_ns_split2}.  The spatial
discretization is done by means of the classical MAC finite
volume stencil.  The accuracy is tested on the following manufactured
solution of the unsteady Stokes equations:
\begin{equation}
\bu= (\sin x \sin(y+t),\cos x \cos(y+t)), \quad p=\cos x \sin(y+t).
\label{analytic}
\end{equation}
and the problem is solved in $\Omega=(0,1){\times}(0,1)$, for $0 \leq t \leq T:=10$ with Dirichlet boundary conditions (given by the pointwise
values of the exact solution).  The initial condition is the exact solution at $t=0$. 
In figure \ref{fig:error1} we present the $L^2$ norm of the errors in the velocity, pressure, and the divergence for the unsteady Stokes equations .  
The results with both schemes are very similar, however,  the equations in \eqref{defect_ns_split0}-\eqref{defect_ns_split2} are much easier to solve 
since all velocity components are decoupled.
\begin{figure}[h]
\centerline{%
\subfigure[$\|\bu-\bw\|_{\bL^2}$ versus $\dt$]{%
\includegraphics[width=0.33\textwidth,bb=20 9 240 238,clip=]{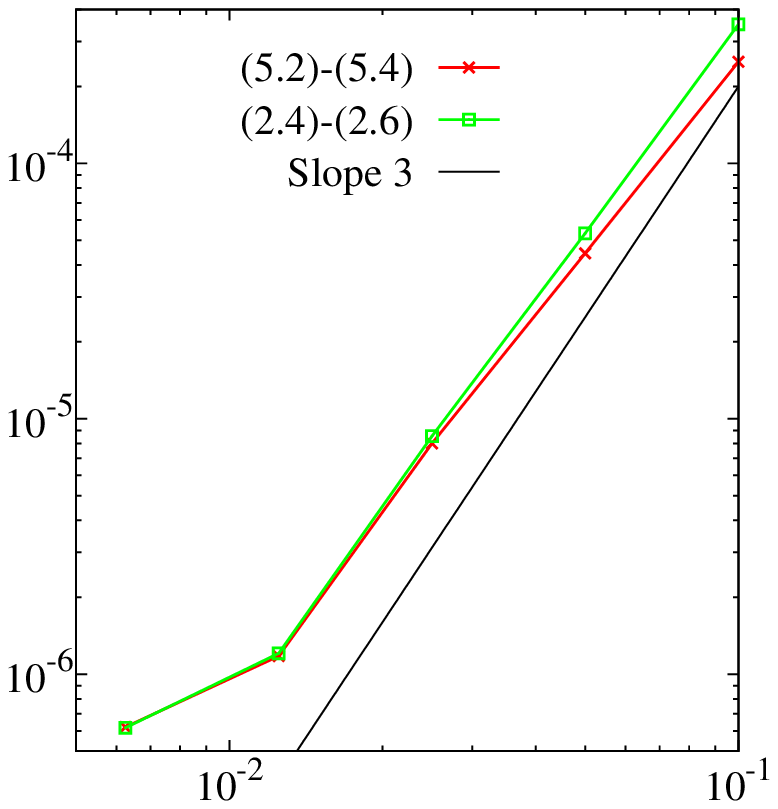}}%
\subfigure[$\|p-q\|_{L^2}$ versus $\dt$]{%
\includegraphics[width=0.33\textwidth,bb=20 9 240 238,clip=]{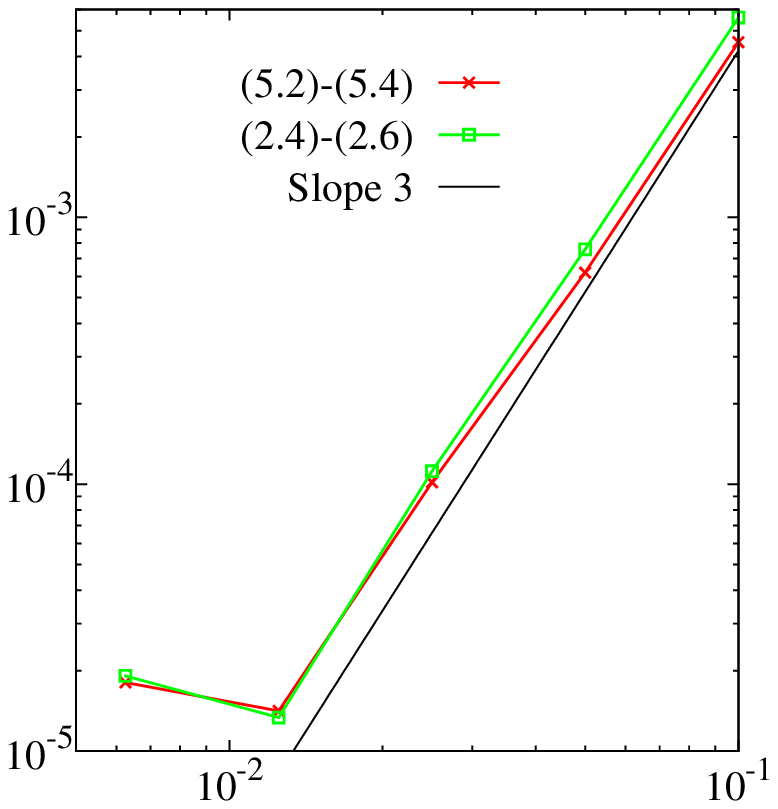}}%
\subfigure[$\|\DIV\bw\|_{L^2}$ versus $\dt$]{%
\includegraphics[width=0.33\textwidth,bb=20 9 240 238,clip=]{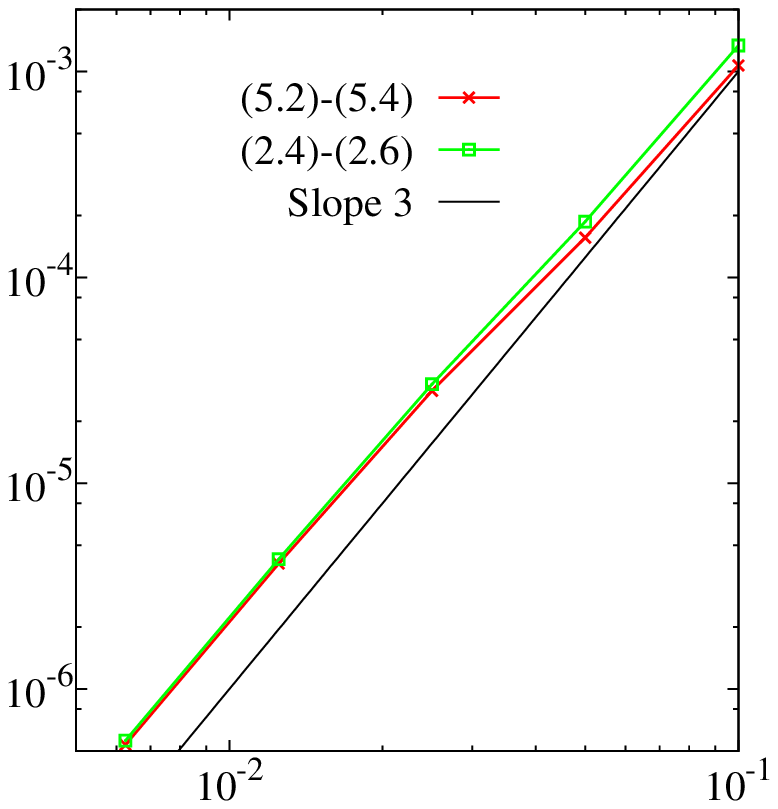}}}
\caption{Log-log plot of the $\bL^2$-norm of the error at
  $T=10$ of the 2D unsteady Stokes equations on a $200\times200$ MAC grid.
  Continuous lines represent the slope 3,  solid lines with {\color{red}{$\times$}} symbols represent the results 
  with \eqref{defect_ns_split0}-\eqref{defect_ns_split2}, solid lines with {\color{green}{$\Box$}} symbols represent the results 
  with \eqref{defect_ns_bootstrap0}-\eqref{defect_ns_bootstrap2}.  }
\label{fig:error1}
\end{figure}

Next, we compare the accuracy of the second order scheme
\eqref{defect_ns_split0}-\eqref{defect_ns_split1} and the second order
direction-splitting bootstrapping scheme
\eqref{eq:split_02}-\eqref{eq:split_12} in figure \ref{fig:error2}.
Although being slightly less accurate in the pressure, the direction
splitting scheme clearly has a good potential since it is less
computationally demanding; we recall that this schemes only requires
the solution of tridiagonal problems and thus can be massively
parallelized as in \cite{GuermondMinevADI}.
\begin{figure}[h]
\centerline{%
\subfigure[$\|\bu-\bw\|_{\bL^2}$ versus $\dt$]{%
\includegraphics[width=0.33\textwidth,bb=20 9 240 238,clip=]{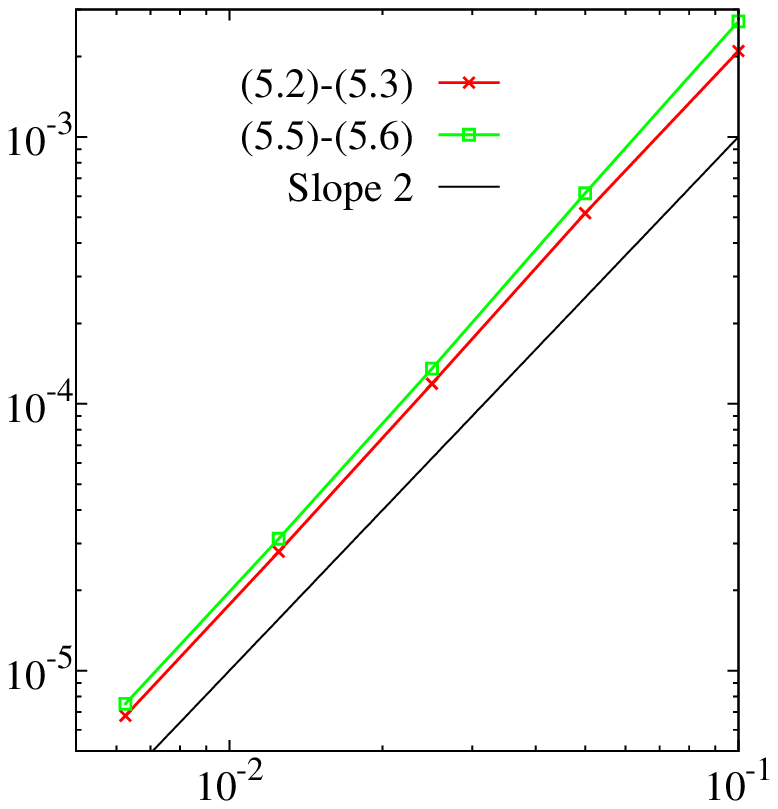}}%
\subfigure[$\|p-q\|_{L^2}$ versus $\dt$]{%
\includegraphics[width=0.33\textwidth,bb=20 9 240 238,clip=]{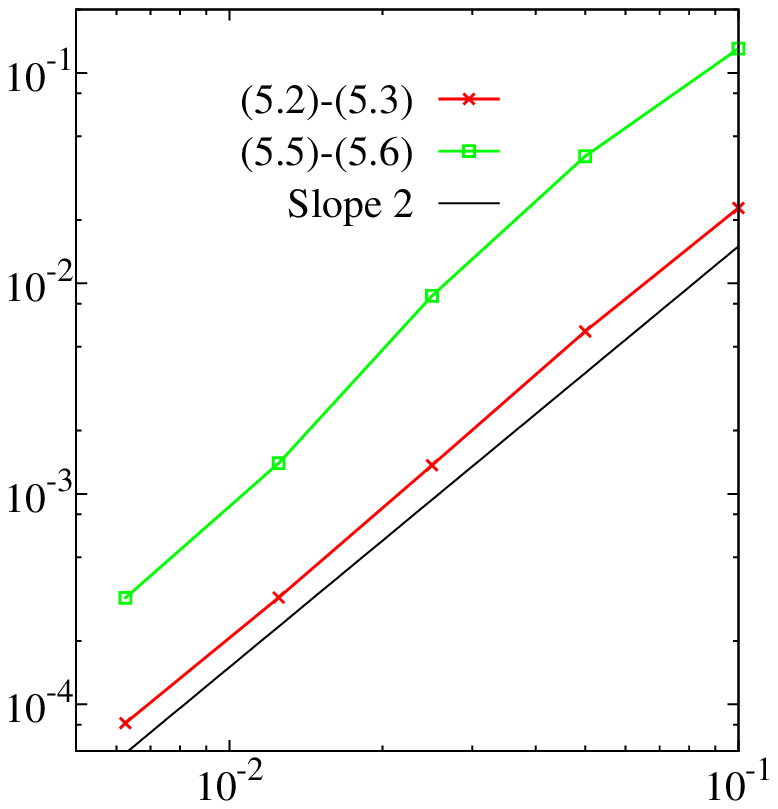}}%
\subfigure[$\|\DIV\bw\|_{L^2}$ versus $\dt$]{%
\includegraphics[width=0.33\textwidth,bb=20 9 240
238,clip=]{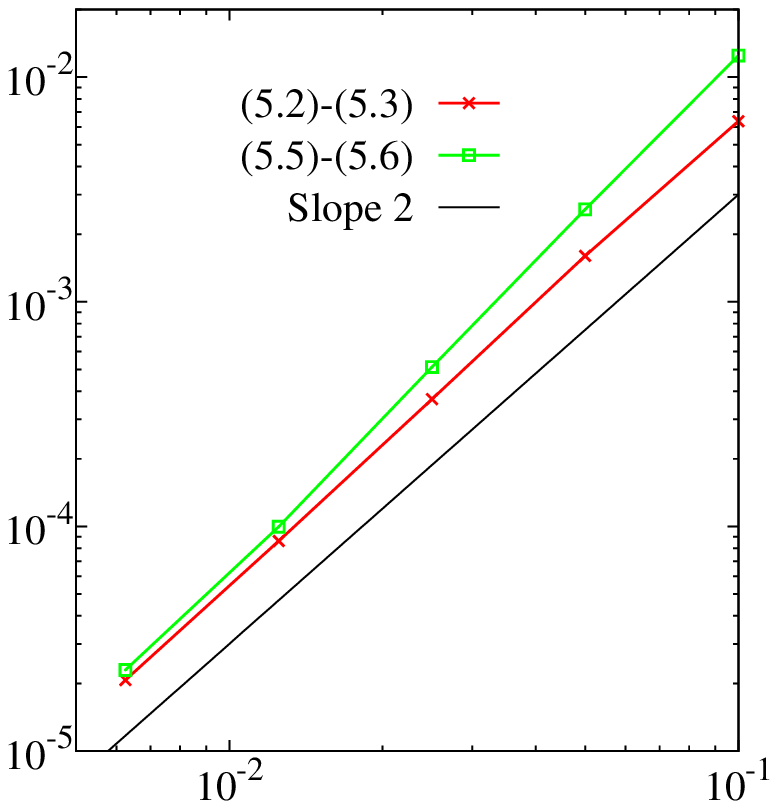}}}
\caption{Log-log plot of the $\bL^2$-norm of the error at $T=10$ of
  the 2D unsteady Stokes equations on a $200\times200$ MAC grid.
  Continuous lines represent the slope 2, solid lines with
  {\color{red}{$\times$}} symbols represent the results with
  \eqref{defect_ns_split0}-\eqref{defect_ns_split1}, solid lines with
  {\color{green}{$\Box$}} symbols represent the results with
  \eqref{eq:split_02}-\eqref{eq:split_12}.  }
\label{fig:error2}
\end{figure}

Finally we present 3D numerical results that demonstrate the accuracy
of \eqref{defect_ns_split0}-\eqref{defect_ns_split1} in the case of
the unsteady Stokes and the Navier-Stokes problem at Re=100. The 3D
manufactured solution is given by:
$u_1=\cos x \sin y \sin(z+t), u_2=\sin x \cos y \sin(z+t), u_3=-2 \sin
x \sin y \cos(z+t), p=\cos(x+y+z+t)$.
The results on a grid of $20\times 20\times 20$ MAC cells are
presented in figure \ref{fig:error3}.  Again, the defect correction
method \eqref{defect_ns_split0}-\eqref{defect_ns_split1} demonstrates
good accuracy and robustness, maintaing stability even at relatively
large time steps $dt=0.1$.
\begin{figure}[h]
\centerline{
\includegraphics[width=0.33\textwidth,bb=20 9 240 238,clip=]{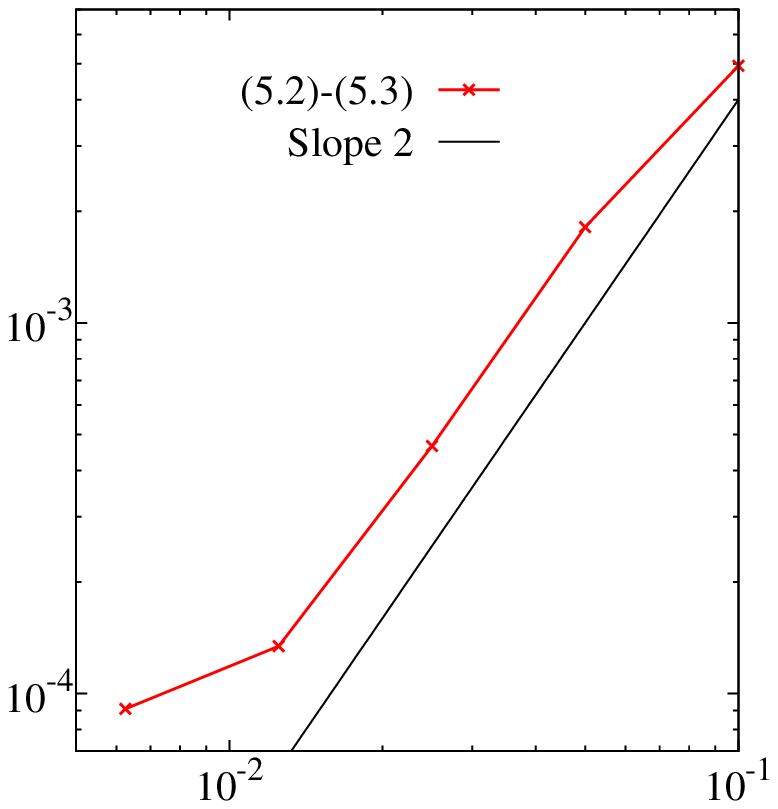}%
\includegraphics[width=0.33\textwidth,bb=20 9 240 238,clip=]{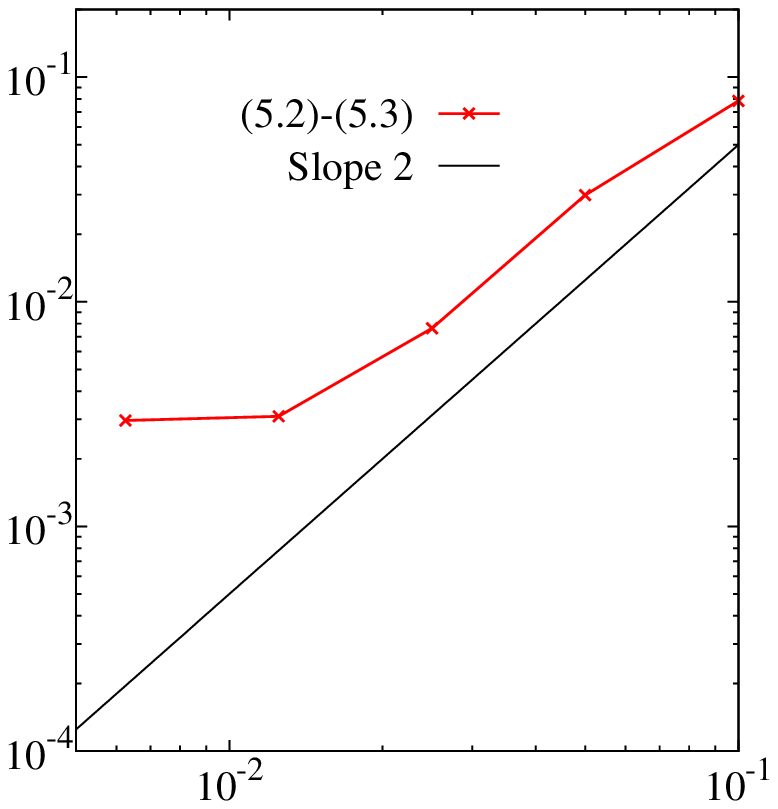}%
\includegraphics[width=0.33\textwidth,bb=20 9 240 238,clip=]{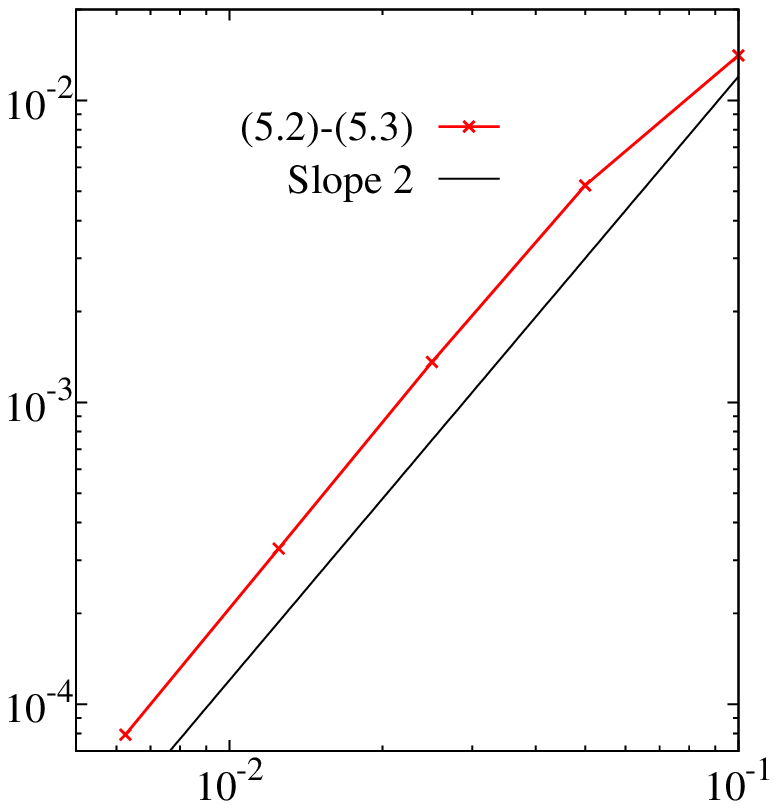}%
}
\centerline{
\subfigure[$\|\bu-\bw\|_{\bL^2}$ versus $\dt$]{%
\includegraphics[width=0.33\textwidth,bb=20 9 240 238,clip=]{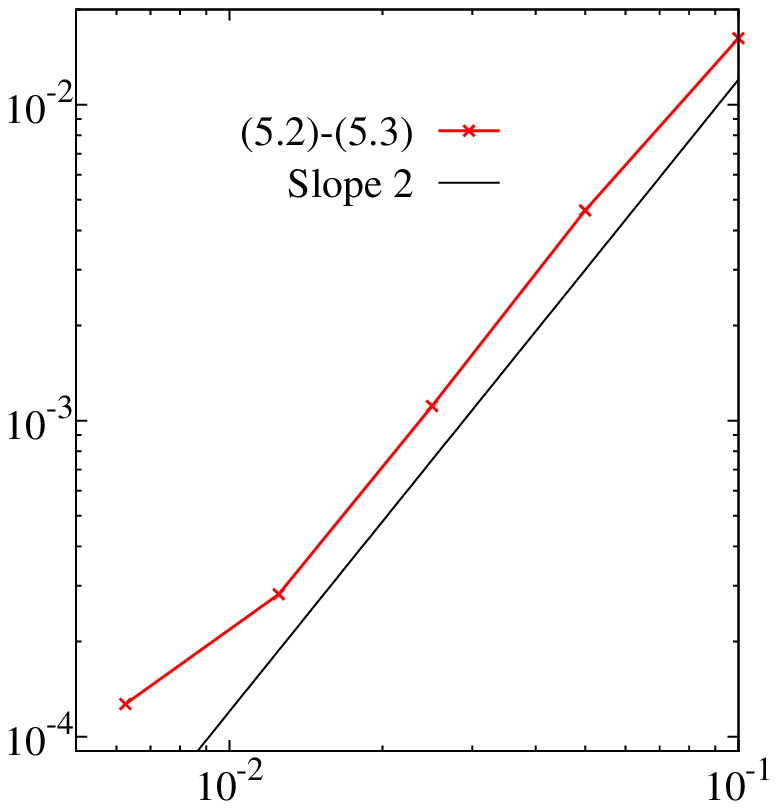}}%
\subfigure[$\|p-q\|_{L^2}$ versus $\dt$]{%
\includegraphics[width=0.33\textwidth,bb=20 9 240 238,clip=]{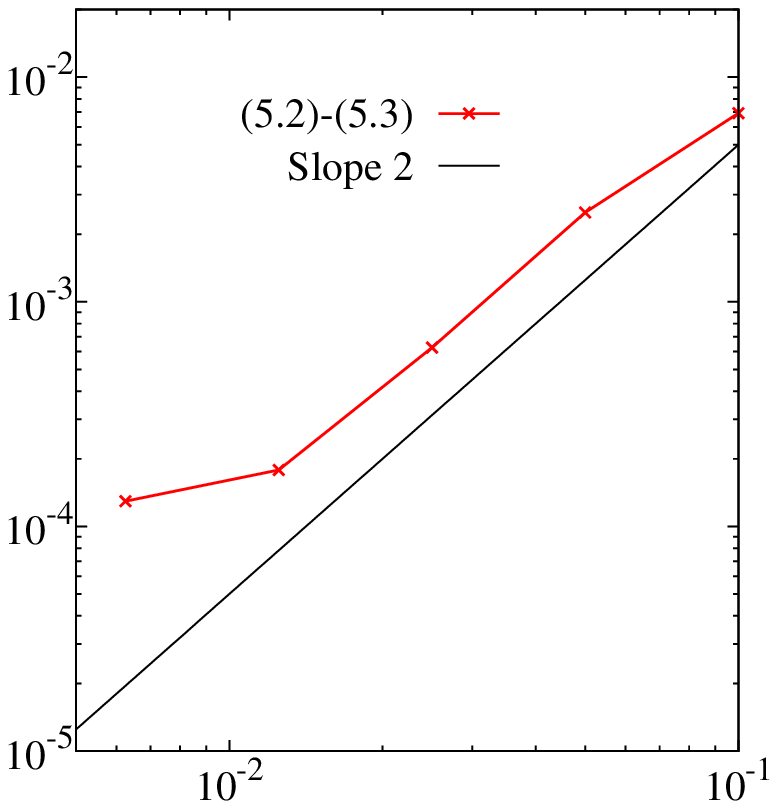}}%
\subfigure[$\|\DIV\bw\|_{L^2}$ versus $\dt$]{%
\includegraphics[width=0.33\textwidth,bb=20 9 240 238,clip=]{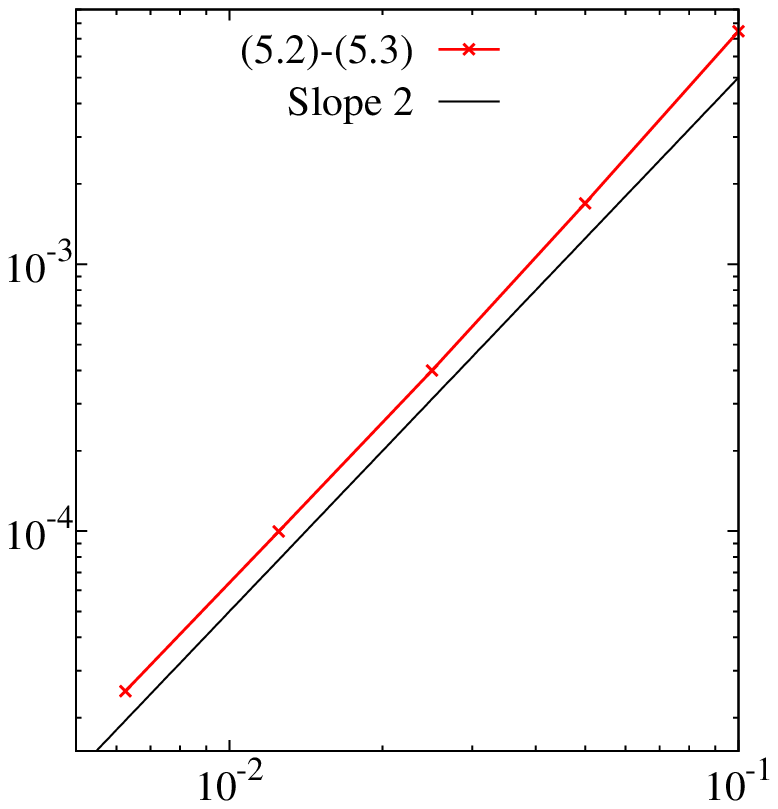}}%
}
\caption{Log-log plot of the $\bL^2$-norm of the error at
  $T=10$ on a $20\times20\times 20$ MAC grid. Top row:  3D unsteady
  Stokes equations, bottom row: 3D Navier-Stokes equations at Re=100.
  Continuous lines represent the slope of 2,  solid lines with {\color{red}{$\times$}} symbols represent the results with \eqref{defect_ns_split0}-\eqref{defect_ns_split1}.  }
\label{fig:error3}
\end{figure}


\section{Conclusions}
In this paper we have revisited the high-order artificial
compressibility methods for incompressible flow of \cite{GM14} and we
have demonstrated that the coupling of the Cartesian components of the
velocity, which is due to the presence of the implicit $\GRAD \DIV$
operator, can be avoided. The resulting schemes thus require only the
solution of a set of classical scalar parabolic problems of the type:
$u_k - \dt \DIV \left(\kappa \GRAD u_k\right)=f$. These schemes can
also be factorized direction-wise to yield computationally very
simple, and yet accurate direction splitting schemes.

When compared to the classical Chorin-Temam-type projection schemes,
the algorithms proposed in this paper are computationally more
efficient since they require the solution of problems with
conditioning scaling like $\dt h^{-2}$
whereas projection methods require the solution of an elliptic
problem for the pressure whose conditioning scales like $h^{-2}$.  In
addition, the present approach allows to develop schemes
of any order in time unlike the projection methods whose accuracy is
limited to second order.

\bibliographystyle{abbrvnat} 
\bibliography{ref}

\begin{thebibliography}{8}
\providecommand{\natexlab}[1]{#1}
\providecommand{\url}[1]{\texttt{#1}}
\expandafter\ifx\csname urlstyle\endcsname\relax
  \providecommand{\doi}[1]{doi: #1}\else
  \providecommand{\doi}{doi: \begingroup \urlstyle{rm}\Url}\fi

\bibitem[Douglas(1962)]{D62}
J.~Douglas, Jr.
\newblock Alternating direction methods for three space variables.
\newblock \emph{Numer. Math.}, 4:\penalty0 41--63, 1962.

\bibitem[Guermond and Minev(2015)]{GM14}
J.-L. Guermond and P.~Minev.
\newblock High-order time stepping for the incompressible {N}avier-{S}tokes
  equations.
\newblock \emph{SIAM J. Sci. Comput.}, 37\penalty0 (6):\penalty0 A2656--A2681,
  2015.

\bibitem[Guermond and Minev(2011)]{GuermondMinevADI}
J.-L. Guermond and P.~D. Minev.
\newblock A new class of massively parallel direction splitting for the
  incompressible {N}avier-{S}tokes equations.
\newblock \emph{Computer Methods in Applied Mechanics and Engineering},
  200\penalty0 (23-24):\penalty0 2083--2093, 2011.

\bibitem[Ladyzhenskaya(1970)]{Ladhyz70b}
O.~A. Ladyzhenskaya.
\newblock \emph{The mathematical theory of viscous incompressible flow (in
  Russian)}.
\newblock Second Russian Edition, revised and extended. Nauka, Moscow, 1970.

\bibitem[Temam(1977)]{Te77}
R.~Temam.
\newblock \emph{{N}avier--{S}tokes Equations}, volume~2 of \emph{Studies in
  Mathematics and its Applications}.
\newblock North-Holland, 1977.

\bibitem[Vabishchevich(2014)]{Vab14}
P.~Vabishchevich.
\newblock \emph{Additive operator-difference schemes: splitting schemes}.
\newblock De Gruyter, Berlin, Boston, 2014.

\bibitem[Vladimirova et~al.(1966)Vladimirova, Kuznetsov, and
  Yanenko]{Yanenko_1966}
N.~Vladimirova, B.~Kuznetsov, and N.~Yanenko.
\newblock Numerical calculation of the symmetrical flow of viscous
  incompressible liquid around a plate (in {R}ussian).
\newblock In \emph{Some Problems in Computational and Applied Mathematics}.
  Nauka, Novosibirsk, 1966.

\bibitem[Yanenko(1971)]{MR46:6613}
N.~N. Yanenko.
\newblock \emph{The method of fractional steps. {T}he solution of problems of
  mathematical physics in several variables}.
\newblock Springer-Verlag, New York, 1971.

\end{thebibliography}
\end{document}